\documentclass{amsart}

\usepackage{amssymb}
\usepackage{hyperref}
\hypersetup{colorlinks=true, citecolor=blue}

\newtheorem{theorem}{Theorem}[section]
\newtheorem{lemma}[theorem]{Lemma}
\newtheorem{proposition}[theorem]{Proposition}

\theoremstyle{remark}
\newtheorem{remark}[theorem]{Remark}

\begin{document}

\title[Harmonic Bergman Spaces on the Real Hyperbolic Ball]
{Harmonic Bergman Spaces on the Real Hyperbolic Ball:
Atomic Decomposition, Interpolation and Inclusion Relations}

\thanks{This research is supported by Eski\c{s}ehir Technical University
Research Fund under grant 22ADP349.}

\author{A. Ers$\dot{\hbox{\i}}$n \"Ureyen}
\address{Department of Mathematics, Faculty of Science,
Eki\c{s}ehir Technical University, 26470, Eski\c{s}ehir,
Turkey}
\email{aeureyen@eskisehir.edu.tr}

\date{\today}



\begin{abstract}
For $\alpha>-1$ and $0<p<\infty$, we study weighted Bergman
spaces $\mathcal B^p_\alpha$ of harmonic functions on the real
hyperbolic ball and obtain an atomic decomposition of these
spaces in terms of reproducing kernels.
We show that an $r$-separated sequence $\{a_m\}$ with
sufficiently large $r$ is an interpolating sequence for
$\mathcal B^p_\alpha$.
Using these we determine precisely when a Bergman space
$\mathcal B^p_\alpha$ is included in another Bergman
space $\mathcal B^q_\beta$.
\end{abstract}

\maketitle

\section{Introduction}\label{SIntro}

For $x,y\in\mathbb R^n$, let $\langle x,y\rangle=x_1y_1+\dots+x_ny_n$ be
the Euclidean inner product and $\lvert x\rvert=\sqrt{\langle x,x\rangle}$
be the corresponding norm.
Let $\mathbb B=\{x\in\mathbb R^n:\lvert x\rvert<1\}$ be the
unit ball and $\mathbb S=\partial\mathbb B$ be the unit sphere.
The hyperbolic ball is $\mathbb B$ endowed with the hyperbolic metric
\begin{equation*}
ds^2=\frac{4}{(1-\lvert x\rvert^2)^2}\sum_{i=1}^n dx_i^2.
\end{equation*}
The Laplacian $\Delta_h$ and the gradient $\nabla^h$ with respect
to the hyperbolic metric are given by (we refer the reader to
\cite[Chapter 3]{St1} for more details about $\Delta_h$ and $\nabla^h$)
\begin{equation*}
(\Delta_h f)(a)=\Delta(f\circ\varphi_a)(0)
\qquad (f\in C^2(\mathbb B)),
\end{equation*}
and
\begin{equation*}
(\nabla^h f)(a)=-\nabla(f\circ\varphi_a)(0)
\qquad (f\in C^1(\mathbb B)),
\end{equation*}
where $\Delta=\partial^2/\partial x_1^2+\dots+\partial^2/\partial x_n^2$
and $\nabla=\big(\partial/\partial x_1,\dots,\partial/\partial x_n\big)$
are the usual Euclidean Laplacian and gradient.
Here $\varphi_a$ is the canonical M\"obius transformation mapping
$\mathbb B$ to $\mathbb B$ and exchanging $a$ and $0$ given
in \eqref{definevarphia}.
It is easy to show that
\begin{equation*}
\Delta_hf(a)=(1-\lvert a\rvert^2)^2\Delta f(a)
+2(n-2)(1-\lvert a\rvert^2)\langle a,\nabla f(a)\rangle,
\end{equation*}
and
\begin{equation}\label{nablahf}
\nabla^hf(a)=(1-\lvert a\rvert^2)\nabla f(a).
\end{equation}
A twice continuously differentiable function
$f\colon\mathbb B\to\mathbb C$ is called hyperbolic harmonic or
$\mathcal H$-harmonic on $\mathbb B$ if $\Delta_hf\equiv 0$.
We denote the space of all $\mathcal H$-harmonic functions
by $\mathcal H(\mathbb B)$.

Let $\nu$ be the Lebesgue measure on $\mathbb R^n$ normalized so
that $\nu(\mathbb B)=1$.
For $\alpha>-1$, define the weighted measure $d\nu_\alpha(x)$ by
\begin{equation*}
d\nu_\alpha(x)=(1-\lvert x\rvert^2)^\alpha\,d\nu(x).
\end{equation*}
For $0<p<\infty$ and $\alpha>-1$, we denote the Lebesgue space with
respect to $d\nu_\alpha$ by $L^p_\alpha=L^p(d\nu_\alpha)$.
The subspace $\mathcal B^p_\alpha$ consisting of $\mathcal H$-harmonic
functions is called the weighted $\mathcal H$-harmonic Bergman space,
\begin{equation*}
\mathcal B^p_\alpha=\Bigl\{f\in\mathcal H(\mathbb B):\|f\|_{L^p_\alpha}^p
=\int_{\mathbb B}\lvert f(x)\rvert^p\,d\nu_\alpha(x)<\infty\Bigr\}.
\end{equation*}
These are Banach spaces when $1\leq p<\infty$, and complete metric spaces
with respect to the metric $d(f,g)=\|f-g\|^p_{L^p_\alpha}$ when $0<p<1$.

Point evaluation functionals are bounded on all $\mathcal B^p_\alpha$
and  in particular $\mathcal B^2_\alpha$ is a reproducing kernel Hilbert space.
Therefore for every $x\in\mathbb B$, there exists
$\mathcal R_\alpha(x,\cdot)\in\mathcal B^2_\alpha$ such that
\begin{equation}\label{reproduce}
f(x)=\int_{\mathbb B} f(y)\overline{\mathcal R_\alpha(x,y)}\,d\nu_\alpha(y)
\qquad (f\in\mathcal B^2_\alpha).
\end{equation}
The reproducing kernel $\mathcal R_\alpha(\cdot,\cdot)$ is symmetric
in its variables, is real valued (so conjugation in \eqref{reproduce}
can be deleted) and is $\mathcal H$-harmonic with respect to each variable.

For $a,b\in\mathbb B$, let $\rho(a,b)=\lvert\varphi_a(b)\rvert$ be the
pseudo-hyperbolic metric, and for $0<r<1$, let
$E_r(a)=\{x\in\mathbb B:\rho(x,a)<r\}$ be the pseudo-hyperbolic ball.
For $0<r<1$, a sequence $\{a_m\}$ of points of $\mathbb B$ is called
$r$-separated if $\rho(a_k,a_m)\geq r$ when $k\ne m$.
An $r$-separated sequence $\{a_m\}$ is called an $r$-lattice if
$\bigcup_{m=1}^\infty E_r(a_m)=\mathbb B$, that is, if $\{a_m\}$ is maximal.

In \cite[Theorem 2]{CR}, it is shown by Coifman and Rochberg that
if $\{a_m\}$ is an $r$-lattice with $r$ sufficiently small, then every
\textit{holomorphic} Bergman function $f\in A^p$ on the unit ball of
$\mathbb C^n$ (more generally on a symmetric Siegel domain of type two) can be
represented in the form $f(z)=\sum_{m=1}^\infty\lambda_m\tilde{B}(z,a_m)$, where
$\{\lambda_m\}\in\ell^p$ and $\tilde{B}(z,a_m)$ is determined by $B(\cdot,a_m)$,
the reproducing kernel at the point $a_m$.
This representation is called atomic decomposition, $B(\cdot,a_m)$ being
the atoms.
They further showed that a similar decomposition holds also for
(Euclidean) \textit{harmonic} functions on the unit ball of $\mathbb R^n$.
This last result is extended in \cite{T1} and \cite{T2} to
harmonic Bergman spaces on bounded symmetric domains of $\mathbb R^n$.

Our first aim in this work is to show that an if $\{a_m\}$ is an
$r$-lattice with small enough $r$, then an analogous series representation
in terms of the reproducing kernels holds also for $\mathcal H$-harmonic
Bergman spaces $\mathcal B^p_\alpha$.
For atomic decomposition of $\mathcal H$-harmonic \textit{Hardy} spaces
on the real hyperbolic ball, see \cite{J}.

\begin{theorem}\label{Tatomic}
Let $\alpha>-1$ and $0<p<\infty$.
Pick $s$ large enough to satisfy
\begin{equation}\label{scond}
\begin{split}
\alpha+1&<p(s+1),\quad\text{if $p\geq 1$}\\
\alpha+n&<p(s+n),\quad\text{if $0<p<1$}.
\end{split}
\end{equation}
There is an $r_0<1/8$ depending only on $n,\alpha,p,s$
such that if $\{a_m\}$ is an $r$-lattice with
$r<r_0$, then for every $f\in\mathcal B^p_\alpha$, there exists
$\{\lambda_m\}\in\ell^p$ such that
\begin{equation}\label{atomFrm}
f(x)=\sum_{m=1}^\infty \lambda_m(1-\lvert a_m\rvert^2)^{s+n-(\alpha+n)/p}\,
\mathcal R_s(x,a_m)\qquad (x\in\mathbb B),
\end{equation}
where the series converges absolutely and uniformly on compact subsets of $\mathbb B$
and in $\|\cdot\|_{\mathcal B^p_\alpha}$, and the norm $\|\{\lambda_m\}\|_{\ell^p}$
is equivalent to the norm $\|f\|_{\mathcal B^p_\alpha}$.
\end{theorem}

\begin{remark}
The decomposition above can be written in other forms.
Theorem \ref{Tatomic} remains true if \eqref{atomFrm} is replaced with
\begin{equation}\label{newad}
f(x)=\sum_{m=1}^\infty \lambda_m
\frac{\mathcal R_s(x,a_m)}{\|\mathcal R_s(\cdot,a_m)\|_{\mathcal B^p_\alpha}}
\qquad (x\in\mathbb B).
\end{equation}
This follows from the fact that
$\|\mathcal R_s(\cdot,a_m)\|_{\mathcal B^p_\alpha}
\sim(1-\lvert a_m\rvert^2)^{(\alpha+n)/p-(s+n)}$
by \eqref{IntRpower} below, and can be verified with only minor changes
in the proof given in Section \ref{SA}.
Also, $\mathcal R_s(a_m,a_m)\sim(1-\lvert a_m\rvert^2)^{-(s+n)}$
by \eqref{EstRtwosided}, and \eqref{atomFrm} can be replaced with
\begin{equation*}
f(x)=\sum_{m=1}^\infty \lambda_m (1-\lvert a_m\rvert^2)^{-(\alpha+n)/p}\,
\frac{\mathcal R_s(x,a_m)}{\mathcal R_s(a_m,a_m)}\qquad (x\in\mathbb B).
\end{equation*}
\end{remark}

We next consider the interpolation problem.
If $\{a_m\}$ is $r$-separated and $f\in\mathcal B^p_\alpha$, then
the sequence (see Proposition \ref{Pfaminlp} below)
\begin{equation*}
\big\{f(a_m)(1-\lvert a_m\rvert^2)^{(\alpha+n)/p}\big\}
\end{equation*}
is in $\ell^p$.
If the converse holds, that is, if for every $\{\lambda_m\}\in\ell^p$,
one can find an $f\in\mathcal B^p_\alpha$ such that
$f(a_m)(1-\lvert a_m\rvert^2)^{(\alpha+n)/p}=\lambda_m$,
then $\{a_m\}$ is called an interpolating sequence for
$\mathcal B^p_\alpha$.
We show that if the separation constant $r$ is large enough,
then $\{a_m\}$ is an interpolating sequence.

\begin{theorem}\label{TInterp}
Let $\alpha>-1$ and $0<p<\infty$.
There is an $r_0$ with $1/2<r_0<1$ depending only on $n,\alpha,p$
such that if $\{a_m\}$ is an $r$-separated
sequence with $r>r_0$, then for every $\{\lambda_m\}\in\ell^p$, there
exists $f\in\mathcal B^p_\alpha$ such that
\begin{equation*}
  f(a_m)(1-\lvert a_m\rvert^2)^{(\alpha+n)/p}=\lambda_m,
\end{equation*}
and the norm $\|f\|_{\mathcal B^p_\alpha}$
is equivalent to the norm $\|\{\lambda_m\}\|_{\ell^p}$.
\end{theorem}

The \textit{holomorphic} analogue of the above theorem is proved in
\cite{Amr} for the unit ball and polydisc, and in \cite{Ro} for
more general domains of $\mathbb C^n$.
For \textit{harmonic} Bergman spaces on the upper half-space
of $\mathbb R^n$, an analogous result is proved in \cite{CY}.

\begin{remark}
Let $b^p_\alpha$ be the weighted (Euclidean) harmonic Bergman spaces
on $\mathbb B$.
The proof of the interpolation theorem we give in Section \ref{SInterp}
below works also for the harmonic case and Theorem \ref{TInterp} is true
when $\mathcal B^p_\alpha$ replaced with $b^p_\alpha$.
\end{remark}

Finally, we determine precisely when a Bergman space $\mathcal B^p_\alpha$
is contained in an another Bergman space $\mathcal B^q_\beta$.

\begin{theorem}\label{Tinclusion}
Let $\alpha,\beta>-1$ and $0<p,q<\infty$.
\begin{enumerate}
\item[(a)] If $q\geq p$, then
\begin{equation*}
  \mathcal B^p_\alpha\subset\mathcal B^q_\beta\quad
  \text{if and only if}\quad
  \frac{\alpha+n}{p}\leq\frac{\beta+n}{q}
\end{equation*}
\item[(b)] If $q<p$, then
\begin{equation*}
  \mathcal B^p_\alpha\subset\mathcal B^q_\beta\quad
  \text{if and only if}\quad
  \frac{\alpha+1}{p}<\frac{\beta+1}{q}
\end{equation*}
\end{enumerate}
In both cases the inclusion
$i\colon\mathcal B^p_\alpha\to\mathcal B^q_\beta$ is continuous.
\end{theorem}

For holomorphic Bergman spaces on the unit ball of $\mathbb C^n$,
a similar theorem has been proved in \cite[Lemma 2.1]{JMT}.
However this source uses gap series formed by using the so-called
Ryll-Wojtaszczyk polynomials (see \cite{RW}).
We do not know whether such type of $\mathcal H$-harmonic functions
exist on the real hyperbolic ball.
Our proof is based on the above atomic decomposition and interpolation
theorems.

\section{Preliminaries}\label{SPre}

In this section we collect some facts about M\"obius
transformations and $\mathcal H$-harmonic Bergman spaces
that will be used in the sequel.
These are all known and we try to provide a suitable
reference.
If the proof is short we include it for the convenience
of the reader.

\subsection{Notation}

We denote positive constants whose exact values are inessential
with $C$.
The value of $C$ may differ from one occurrence to another.
We write $X\lesssim Y$ if $X\leq CY$, and $X\sim Y$ if both
$X\leq CY$ and $Y\leq CX$.

For $x,y\in\mathbb R^n$, we define
\begin{equation*}
[x,y]:=\sqrt{1-2\langle x,y\rangle+|x|^2|y|^2}.
\end{equation*}
It is symmetric in its variables and the following equality holds
\begin{equation}\label{xysquare}
[x,y]^2=\lvert x-y\rvert^2+(1-\lvert x\rvert^2)(1-\lvert y\rvert^2).
\end{equation}
If either of the variables is $0$, then $[x,0]=[0,y]=1$; otherwise
\begin{equation}\label{bracketxy}
[x,y]=\Bigl\lvert\lvert y\rvert x-\frac{y}{\lvert y\rvert}\Bigr\rvert
=\Bigl\lvert\frac{x}{\lvert x\rvert}-\lvert x\rvert y\Bigr\rvert,
\end{equation}
and so
\begin{equation}\label{xybig}
1-\lvert x\rvert\lvert y\rvert
\leq [x,y]
\leq 1+\lvert x\rvert\lvert y\rvert
\qquad (x,y\in\mathbb B).
\end{equation}

\subsection{M\"{o}bius transformations and pseudo-hyperbolic metric}

For details about the facts listed in this subsection we refer
the reader to \cite{Ah} or \cite{St1}.

A M\"{o}bius transformation of $\hat{\mathbb R}^n=\mathbb R^n\cup\{\infty\}$ is a
finite composition of reflections (inversions) in spheres or planes.
We denote the group of all M\"{o}bius transformations
mapping $\mathbb B$ to $\mathbb B$ by $\mathcal M(\mathbb B)$.
For $a\in\mathbb B$, the mapping
\begin{equation}\label{definevarphia}
\varphi_a(x)=\frac{a\lvert x-a\rvert^2+(1-\lvert a\rvert^2)(a-x)}{[x,a]^2}
\qquad (x\in\mathbb B)
\end{equation}
is in $\mathcal M(\mathbb B)$, exchanges $a$ and $0$, and
satisfies $\varphi_a\circ\varphi_a=\text{Id}$.
The group $\mathcal M(\mathbb B)$ is generated by
$\{\varphi_a:a\in\mathbb B\}$ and orthogonal transformations.
A very useful identity involving $\varphi_a$ is
(\cite[Eqn. 2.1.7]{St1})
\begin{equation}\label{MobiusIdnt}
1-\lvert\varphi_a(x)\rvert^2
=\frac{(1-\lvert a\rvert^2)(1-\lvert x\rvert^2)}{[x,a]^2},
\end{equation}
and the Jacobian $J_{\varphi_a}$ satisfies (\cite[Theorem 3.3.1]{St1})
\begin{equation}\label{Jacobian}
\lvert J_{\varphi_a}(x)\rvert
=\frac{(1-\lvert\varphi_a(x)\rvert^2)^n}{(1-\lvert x\rvert^2)^n}.
\end{equation}

\begin{lemma}\label{La_varphia}
For $a,x\in\mathbb B$, the following equality holds
\begin{equation*}
[\varphi_a(x),a]=\frac{1-\lvert a\rvert^2}{[x,a]}.
\end{equation*}
\end{lemma}

\begin{proof}
Replacing $x$ in \eqref{MobiusIdnt} with $\varphi_a(x)$ and noting that
$\varphi_a\circ\varphi_a=\text{Id}$ shows
\begin{equation*}
[\varphi_a(x),a]^2
=\frac{(1-\lvert a\rvert^2)(1-\lvert \varphi_a(x)\rvert^2)}{1-\lvert x\rvert^2}.
\end{equation*}
Using \eqref{MobiusIdnt} again gives the result.
\end{proof}

For $a,b\in\mathbb B$, the pseudo-hyperbolic metric
$\rho(a,b)=\lvert\varphi_a(b)\rvert$ satisfies
\begin{equation}\label{pseudo}
\rho(a,b)=\frac{\lvert a-b\rvert}{[a,b]},
\end{equation}
by \eqref{MobiusIdnt} and \eqref{xysquare}.
It is bounded $0\leq \rho(a,b)<1$, and is M\"{o}bius invariant
in the sense that $\rho(\psi(a),\psi(b))=\rho(a,b)$
for every $\psi\in\mathcal M(\mathbb B)$.
It satisfies not only the triangle
inequality (\cite[Theorem 2.2.3]{St1}), but the following strong
triangle inequality.
\begin{lemma}\label{LstrongTri}
For $a,b,x\in\mathbb B$, the following inequalities hold
\begin{equation*}
\frac{\big\lvert\rho(a,x)-\rho(b,x)\big\rvert}{1-\rho(a,x)\rho(b,x)}
\leq \rho(a,b)
\leq\frac{\rho(a,x)+\rho(b,x)}{1+\rho(a,x)\rho(b,x)}.
\end{equation*}
\end{lemma}

\begin{proof}
By \eqref{MobiusIdnt} and \eqref{xybig}, we have
\begin{equation*}
\frac{(1-\lvert a\rvert^2)(1-\lvert b\rvert^2)}
{(1+\lvert a\rvert\lvert b\rvert)^2}
\leq 1-\lvert\varphi_a(b)\rvert^2
\leq\frac{(1-\lvert a\rvert^2)(1-\lvert b\rvert^2)}
{(1-\lvert a\rvert\lvert b\rvert)^2},
\end{equation*}
which after simplification implies
\begin{equation*}
\frac{\big\lvert\lvert a\rvert-\lvert b\rvert\big\rvert}
{1-\lvert a\rvert\lvert b\rvert}
\leq \lvert\varphi_a(b)\rvert
\leq\frac{\lvert a\rvert+\lvert b\rvert}
{1+\lvert a\rvert\lvert b\rvert}.
\end{equation*}
Since $\rho(a,b)=\lvert\varphi_a(b)\rvert$ and $\rho(a,0)=\lvert a\rvert$,
$\rho(b,0)=\lvert b\rvert$, this proves the lemma when $x=0$.
To obtain the general case apply this special case with $a,b$ replaced
by $\varphi_x(a),\varphi_x(b)$, write $0=\varphi_x(x)$, and
use the M\"{o}bius invariance of $\rho$.
\end{proof}

\begin{lemma}\label{Lratiobnd}
For $x,y\in\mathbb B$,
\begin{equation*}
1-\rho(x,y)
\leq\frac{1-\lvert x\rvert^2}{[x,y]}
\leq 1+\rho(x,y).
\end{equation*}
\end{lemma}

\begin{proof}
The lemma clearly holds when $x=0$.
Otherwise, let $x^*:=x/\lvert x\rvert^2$ be the inversion of
$x$ with respect to the unit sphere $\mathbb S$.
Multiply the triangle inequality
\begin{equation*}
\lvert x^*-y\rvert-\lvert y-x\rvert
\leq \lvert x^*-x\rvert
\leq\lvert x^*-y\rvert+\lvert y-x\rvert
\end{equation*}
with $\lvert x\rvert$.
Noting that
$\lvert x\rvert\lvert x^*-y\rvert=[x,y]$
by \eqref{bracketxy}, and
$\lvert x\rvert\lvert x^*-x\rvert=1-\lvert x\rvert^2$,
we deduce
\begin{equation*}
[x,y]-\lvert x\rvert\lvert y-x\rvert
\leq 1-\lvert x\rvert^2
\leq [x,y]+\lvert x\rvert\lvert y-x\rvert.
\end{equation*}
The lemma follows from the facts that
$\lvert y-x\rvert=\rho(x,y)[x,y]$ by \eqref{pseudo},
and $\lvert x\rvert<1$.
\end{proof}

The following lemma is a slight modification of \cite[Lemma 2.1]{CKL} and
immediately follows from Lemma \ref{Lratiobnd}.

\begin{lemma}\label{Lratiosquare}
For $x,y\in\mathbb B$,
\begin{equation*}
\frac{1-\rho(x,y)}{1+\rho(x,y)}
\leq\frac{1-\lvert x\rvert^2}{1-\lvert y\rvert^2}
\leq\frac{1+\rho(x,y)}{1-\rho(x,y)}.
\end{equation*}
\end{lemma}

The next lemma is proved in \cite[Lemma 2.2]{CKL}.

\begin{lemma}\label{Lratiobracket}
For $a,x,y\in\mathbb B$,
\begin{equation*}
\frac{1-\rho(x,y)}{1+\rho(x,y)}
\leq\frac{[x,a]}{[y,a]}
\leq\frac{1+\rho(x,y)}{1-\rho(x,y)}.
\end{equation*}
\end{lemma}

Let $\mathbb B_r=\{x\in\mathbb R^n:\lvert x\rvert<r\}$.
The pseudo-hyperbolic ball
$E_r(a)=\{x\in\mathbb B:\rho(x,a)<r\}=\varphi_a(\mathbb B_r)$ is
also a Euclidean ball with (see \cite[Theorem 2.2.2]{St1})
\begin{equation}\label{PseudoBall}
\text{center}=\frac{(1-r^2)a}{1-\lvert a\rvert^2r^2}
\quad \text{and}\quad
\text{radius}=\frac{(1-\lvert a\rvert^2)r}{1-\lvert a\rvert^2r^2}.
\end{equation}
The measure $d\tau(x)=(1-\lvert x\rvert^2)^{-n}d\nu(x)$
is the invariant measure on $\mathbb B$ and
the volume $\tau(E_r(a))=\tau(\mathbb B_r)$ is independent
of $a\in\mathbb B$.

\subsection{Separated sequences and lattices}

A sequence $\{a_m\}$ of points of $\mathbb B$ is called
$r$-separated if $\rho(a_k,a_m)\geq r$ for every $k\ne m$.
If, in addition, $\bigcup_{m=1}^\infty E_r(a_m)=\mathbb B$, then
$\{a_m\}$ is called an $r$-lattice.
As is explained in \cite[p.~18]{CR}, for every $0<r<1$, there
exists an $r$-lattice and every $r$-separated sequence can be
completed to an $r$-lattice.

\begin{lemma}\label{LEm}
Let $\{a_m\}$ be an $r$-lattice.
There exists a sequence $\{E_m\}$ of disjoint sets such that
$\bigcup_{m=1}^\infty E_m=\mathbb B$ and
\begin{equation}\label{Emsubset}
  E_{r/2}(a_m)\subset E_m\subset E_r(a_m).
\end{equation}
\end{lemma}

\begin{proof}
Let $E_1=E_r(a_1)\backslash\bigcup_{m=2}^\infty E_{r/2}(a_m)$
and given $E_1,\dots,E_{m-1}$, let
\begin{equation*}
E_m=E_r(a_m)\backslash\biggl(\bigcup_{i=1}^{m-1} E_i\
\bigcup\bigcup_{i=m+1}^\infty E_{r/2}(a_m)\biggr).\qedhere
\end{equation*}
\end{proof}

The next lemma follows from a standard invariant volume argument.

\begin{lemma}\label{LNballs}
Let $0<r,\delta<1$ and $\{a_m\}$ be $r$-separated.
There exists $N$ depending only on $n,r,\delta$ such that
every $x\in\mathbb B$ belongs to at most $N$ of the balls
$E_\delta(a_m)$.
\end{lemma}

\begin{proof}
Suppose $x\in\mathbb B$ belongs to $N_x$ of the balls $E_\delta(a_m)$.
Set
\begin{equation*}
s:=\frac{\delta+r/2}{1+\delta r/2}<1,
\end{equation*}
where the last inequality follows from $(1-\delta)(1-r/2)>0$.
If $\rho(x,a_m)<\delta$, then by Lemma \ref{LstrongTri},
$E_{r/2}(a_m)\subset E_x(s)$.
Since the balls $E_{r/2}(a_m)$ are disjoint and
$\tau$ is the invariant measure, it follows that
$N_x\,\tau(\mathbb B_{r/2})\leq\tau(\mathbb B_s)$.
\end{proof}

\begin{lemma}\label{LSumam}
Let $\gamma\in\mathbb R$ and $0<r<1$.
\begin{enumerate}
\item[(a)] If $\{a_m\}$ is $r$-separated and $\gamma>n-1$, then
$\sum_{m=1}^\infty (1-\lvert a_m\rvert^2)^\gamma<\infty$.
\item[(b)] If $\{a_m\}$ is an $r$-lattice, then
$\sum_{m=1}^\infty (1-\lvert a_m\rvert^2)^\gamma<\infty$ if and only if
$\gamma>n-1$.
\end{enumerate}
\end{lemma}

\begin{proof}
To see part (a), we note that
\begin{equation}\label{Er2sim}
\int_{E_{r/2}(a_m)}(1-\lvert y\rvert^2)^{\gamma-n}\,d\nu(y)
\sim(1-\lvert a_m\rvert^2)^\gamma,
\end{equation}
where the implied constants depend only on the fixed parameters
$n,\gamma,r$ and are independent of $a_m$.
This is true because for $y\in E_{r/2}(a_m)$, we have
$(1-\lvert y\rvert^2)\sim(1-\lvert a_m\rvert^2)$ by
Lemma \ref{Lratiosquare} and
$\nu(E_{r/2}(a_m))\sim (1-\lvert a_m\rvert^2)^n$ by \eqref{PseudoBall}.
Since the balls $E_{r/2}(a_m)$ are disjoint, we obtain that
\begin{equation*}
\sum_{m=1}^\infty (1-\lvert a_m\rvert^2)^\gamma
\lesssim\int_{\mathbb B} (1-\lvert y\rvert^2)^{\gamma-n}\,d\nu(y).
\end{equation*}
If $\gamma>n-1$, then the above integral is finite.

For part (b), let $E_m$ be as given in Lemma \ref{LEm}.
By \eqref{Emsubset}, we similarly have
\begin{equation}\label{Emsim}
\int_{E_m}(1-\lvert y\rvert^2)^{\gamma-n}\,d\nu(y)
\sim (1-\lvert a_m\rvert^2)^\gamma
\end{equation}
and therefore
$$\sum_{m=1}^\infty (1-\lvert a_m\rvert^2)^\gamma
\sim \sum_{m=1}^\infty\int_{E_m}(1-\lvert y\rvert^2)^{\gamma-n}\,d\nu(y)
=\int_{\mathbb B} (1-\lvert y\rvert^2)^{\gamma-n}\,d\nu(y).$$
The last integral is finite if and only if $\gamma>n-1$.
\end{proof}

\subsection{Estimates of reproducing kernels and projection}

The following upper estimates of the reproducing kernels
$\mathcal R_\alpha$ of $\mathcal H$-harmonic Bergman spaces
are proved in \cite[Theorem 1.2]{U}.
\begin{lemma}\label{LEstR}
For $\alpha>-1$, there exists a constant $C=C(\alpha,n)$ such that
for every $x,y\in\mathbb B$,
\begin{enumerate}
  \item[(a)] $\lvert\mathcal R_\alpha(x,y)\rvert\leq\dfrac{C}{[x,y]^{\alpha+n}}$,
  \item[(b)] $\lvert\nabla_1 \mathcal R_\alpha(x,y)\rvert
  \leq\dfrac{C}{[x,y]^{\alpha+n+1}}$.
\end{enumerate}
Here $\nabla_1$ means the gradient is taken with respect to the
first variable.
\end{lemma}

On the diagonal $y=x$, the following two-sided estimate
holds (\cite[Lemma 6.1]{U}).
\begin{equation}\label{EstRtwosided}
\mathcal R_\alpha(x,x)\sim\frac{1}{(1-\lvert x\rvert^2)^{\alpha+n}}.
\end{equation}
The following estimate of the integrals of powers of the
reproducing kernels is part of \cite[Theorem 1.3]{U}.
If $\alpha,s>-1$, $0<p<\infty$ and $p(s+n)-(\alpha+n)>0$, then
\begin{equation}\label{IntRpower}
\int_\mathbb B\bigl\lvert\mathcal R_s(x,y)\bigr\rvert^{p}\,d\nu_\alpha(y)
\sim\dfrac{1}{(1-|x|^2)^{p(s+n)-(\alpha+n)}}.
\end{equation}

For $s>-1$, we define the projection operator $P_s$ by
\begin{equation*}
P_s f(x)=\int_{\mathbb B}f(y)\mathcal R_s(x,y)\,d\nu_s(y),
\end{equation*}
and the related operators $\tilde{P}_s$ and $Q_s$ by
\begin{equation*}
\tilde{P}_s f(x)
=\int_{\mathbb B}f(y)\,\lvert\mathcal R_s(x,y)\rvert\,d\nu_s(y),
\end{equation*}
and
\begin{equation}\label{DefineQs}
Q_s f(x)=\int_{\mathbb B}\frac{f(y)}{[x,y]^{s+n}}\,d\nu_s(y),
\end{equation}
for suitable $f$.

\begin{lemma}\label{Lprojection}
Let $1\leq p<\infty$ and $\alpha,s>-1$.
The following are equivalent:
\begin{enumerate}
\item[(a)] $P_s$ is bounded on $L^p_\alpha$,
\item[(b)] $\tilde{P}_s$ is bounded on $L^p_\alpha$,
\item[(c)] $Q_s$ is bounded on $L^p_\alpha$,
\item[(d)] $\alpha+1<p(s+1)$.
\end{enumerate}
In case (d) holds, then $P_sf=f$ for every $f\in\mathcal B^p_\alpha$.
\end{lemma}

\begin{proof}
(c) $\Rightarrow$ (b) follows from part (a) of Lemma \ref{LEstR},
(b) $\Rightarrow$ (a) is clear,
(a) $\Rightarrow$ (d) is proved in \cite[Theorem 1.1]{U},
and (d) $\Rightarrow$ (c) is well-known and included in the proof of
\cite[Theorem 1.1]{U}.
\end{proof}

We denote by $\sigma$ the surface measure on $\mathbb S$
normalized so that $\sigma(\mathbb S)=1$.
For a proof of the next estimate see \cite[Proposition 2.2]{LS}.

\begin{lemma}\label{LIntonSandB}
Let $b>-1$ and $c\in\mathbb R$.
For $x\in\mathbb B$, define
\begin{equation*}
I_c(x):=\int_{\mathbb S}\frac{d\sigma(\zeta)}{\lvert x-\zeta\rvert^{n-1+c}}
\qquad\text{and}\qquad
J_{b,c}(x)
:=\int_{\mathbb B}\frac{(1-\lvert y\rvert^2)^b}{[x,y]^{n+b+c}}\,d\nu(y).
\end{equation*}
Both $I_c(x)$ and $J_{b,c}(x)$ have the same upper and lower estimates for
all $x\in\mathbb B$,
\begin{equation*}
I_c(x)\sim J_{b,c}(x)\sim
\begin{cases}
\dfrac{1}{(1-\lvert x\rvert^2)^{c}},&\text{if $c>0$};\\
1+\log\dfrac{1}{1-\lvert x\rvert^2},&\text{if $c=0$};\\
1,&\text{if $c<0$}.
\end{cases}
\end{equation*}
The implied constants depend only on $n,b,c$, and are independent of $x$.
\end{lemma}

We record the following elementary facts about the sequence spaces
for future reference.

\begin{lemma}\label{Lseq}
(i) For $0<p<q<\infty$, $\|\{\lambda_m\}\|_{\ell^q}\leq\|\{\lambda_m\}\|_{\ell^p}$.

(ii) Let $1<p<\infty$ and $p'$ be the conjugate exponent of $p$, $1/p+1/p'=1$.
If $\sum_{m=1}^\infty\lvert \lambda_m\beta_m\rvert<\infty$ for every $\{\beta_m\}\in\ell^{p'}$,
then $\{\lambda_m\}\in\ell^p$.
\end{lemma}

\section{Atomic Decomposition}\label{SA}

The purpose of this section is to prove Theorem \ref{Tatomic}.
The main problem is to show that the operator $U$ defined in
\eqref{defineU} below is onto under the assumptions of the theorem and
we do this through a couple of propositions.

\begin{proposition}\label{Patom1}
For $\alpha>-1$ and $0<p<\infty$, choose $s$ so that \eqref{scond} holds
and let $0<r<1$.
If $\{a_m\}$ is $r$-separated, then the operator
$U\colon\ell^p\to\mathcal B^p_\alpha$ mapping $\lambda=\{\lambda_m\}$ to
\begin{equation}\label{defineU}
U\lambda(x)=\sum_{m=1}^\infty\lambda_m(1-\lvert a_m\rvert^2)^{s+n-(\alpha+n)/p}\,
\mathcal R_s(x,a_m)\qquad (x\in\mathbb B)
\end{equation}
is bounded.
The above series converges absolutely and uniformly on compact subsets of $\mathbb B$,
and also in $\|\cdot\|_{\mathcal B^p_\alpha}$.
\end{proposition}

\begin{proof}
Throughout the proof we suppress the constants that depend on the
fixed parameters $n,\alpha,p,s$ and $r$.

We begin with the case $0<p\leq1$.
We first show that for $\lambda\in\ell^p$, the series in \eqref{defineU}
converges absolutely and uniformly on compact subsets of $\mathbb B$ which
will imply that $U\lambda$ is $\mathcal H$-harmonic on $\mathbb B$
since each $\mathcal R_s(\cdot,a_m)$ is $\mathcal H$-harmonic.
By Lemma \ref{LEstR} (a), if $\lvert x\rvert\leq R<1$, then
$\lvert\mathcal R_s(x,a_m)\rvert\lesssim 1$ since $[x,a_m]\geq 1-\lvert x\rvert$
by \eqref{xybig} and using also Lemma \ref{Lseq} (i) we deduce
\begin{multline*}
\sum_{m=1}^\infty\lvert\lambda_m\rvert(1-\lvert a_m\rvert^2)^{s+n-(\alpha+n)/p}\,
\lvert\mathcal R_s(x,a_m)\rvert
\lesssim\sum_{m=1}^\infty\lvert\lambda_m\rvert(1-\lvert a_m\rvert^2)^{s+n-(\alpha+n)/p}\\
\leq\Bigl(\sum_{m=1}^\infty
\lvert\lambda_m\rvert^p(1-\lvert a_m\rvert^2)^{p(s+n)-(\alpha+n)}\Bigr)^{1/p}
\leq\Bigl(\sum_{m=1}^\infty\lvert\lambda_m\rvert^p\Bigr)^{1/p},
\end{multline*}
where in the last inequality we use $p(s+n)-(\alpha+n)>0$ by \eqref{scond}.

To show $\|U\lambda\|_{B^p_\alpha}\lesssim\|\lambda\|_{\ell^p}$, we again
use Lemma \ref{Lseq} (i) and obtain
\begin{align*}
\|U\lambda\|_{\mathcal B^p_\alpha}^p&=\int_\mathbb B\biggl\lvert
\sum_{m=1}^\infty\lambda_m(1-\lvert a_m\rvert^2)^{s+n-(\alpha+n)/p}\,
\mathcal R_s(x,a_m)\biggr\rvert^p\,d\nu_\alpha(x)\\
&\leq\sum_{m=1}^\infty\lvert\lambda_m\rvert^p(1-\lvert a_m\rvert^2)^{p(s+n)-(\alpha+n)}
\int_\mathbb B\lvert\mathcal R_s(x,a_m)\rvert^p\,d\nu_\alpha(x)
\lesssim\sum_{m=1}^\infty\lvert\lambda_m\rvert^p,
\end{align*}
where the last inequality follows from \eqref{IntRpower} by \eqref{scond}.
This also shows that the series in \eqref{defineU} converges
in the norm $\|\cdot\|_{\mathcal B^p_\alpha}$.

We now consider the case $1<p<\infty$.
Let $p'$ be the conjugate exponent of $p$.
The series in \eqref{defineU} converges absolutely and uniformly
on compact subsets of $\mathbb B$ because we again have
$\lvert\mathcal R_s(x,a_m)\rvert\lesssim 1$ and by H\"older's inequality,
\begin{equation*}
\sum_{m=1}^\infty\lvert\lambda_m\rvert(1-\lvert a_m\rvert^2)^{s+n-(\alpha+n)/p}
\leq \|\{\lambda_m\}\|_{\ell^p}
\Big(\sum_{m=1}^\infty(1-\lvert a_m\rvert^2)^{p'(s+n-(\alpha+n)/p)}\Bigr)^{1/p'}.
\end{equation*}
The last sum is finite by Lemma \ref{LSumam} since the inequality
$p'(s+n-(\alpha+n)/p)>n-1$ is equivalent to \eqref{scond}.
Thus $U\lambda$ is $\mathcal H$-harmonic on $\mathbb B$.

To show $\|U\lambda\|_{B^p_\alpha}\lesssim\|\lambda\|_{\ell^p}$,
following \cite{CR}, \cite{CL} and \cite{T2}, we use the projection
theorem.
Let
\begin{equation*}
g(x):=\sum_{m=1}^\infty\lvert\lambda_m\rvert
(1-\lvert a_m\rvert^2)^{-(\alpha+n)/p}\,
\chi_{E_{r/2}(a_m)}(x)
\qquad (x\in\mathbb B),
\end{equation*}
where $\chi_{E_{r/2}(a_m)}$ is the characteristic function of the
set $E_{r/2}(a_m)$.
We have $\|g\|_{L^p_\alpha}\sim\|\lambda\|_{\ell^p}$ since
by \eqref{Er2sim},
\begin{equation*}
\|g\|_{L^p_\alpha}^p
=\sum_{m=1}^\infty\lvert\lambda_m\rvert^p
(1-\lvert a_m\rvert^2)^{-(\alpha+n)}\nu_\alpha(E_{r/2}(a_m))
\sim\sum_{m=1}^\infty\lvert\lambda_m\rvert^p.
\end{equation*}
Next, with $Q_s$ as in \eqref{DefineQs},
\begin{align*}
Q_s g(x)=\int_{\mathbb B}\frac{g(y)\,d\nu_s(y)}{[x,y]^{s+n}}
&=\sum_{m=1}^\infty\lvert\lambda_m\rvert(1-\lvert a_m\rvert^2)^{-(\alpha+n)/p}
\int_{E_{r/2}(a_m)}\frac{(1-\vert y\rvert^2)^s}{[x,y]^{s+n}}\,d\nu(y)\\
&\sim\sum_{m=1}^\infty
\frac{\lvert\lambda_m\rvert(1-\lvert a_m\rvert^2)^{s+n-(\alpha+n)/p}}{[x,a_m]^{s+n}},
\end{align*}
where in the last line we use the fact that $[x,y]\sim [x,a_m]$
for $y\in E_{r/2}(a_m)$ by Lemma \ref{Lratiobracket},
and then \eqref{Er2sim}.
This shows that $\lvert U\lambda(x)\rvert\lesssim Q_sg(x)$
by Lemma \ref{LEstR} (a).
Since $Q_s$ is bounded by Lemma \ref{Lprojection} and \eqref{scond},
we conclude
\begin{equation*}
\|U\lambda\|_{\mathcal B^p_\alpha}\lesssim\|Q_s g\|_{L^p_\alpha}
\lesssim\|g\|_{L^p_\alpha}\sim\|\lambda\|_{\ell^p}.\qedhere
\end{equation*}
\end{proof}

To verify that the above operator $U\colon\ell^p\to\mathcal B^p_\alpha$ is
onto under the additional assumption that $\{a_m\}$ is an $r$-lattice
with $r$ small enough, we need to consider a second operator.
We first recall the following sub-mean value inequality for
$\mathcal H$-harmonic functions.
For a proof, see \cite[Section 4.7]{St1}.

\begin{lemma}\label{Lsubmean}
Let $f\in\mathcal H(\mathbb B)$ and $0<p<\infty$.
Then for all $a\in\mathbb B$ and all $0<\delta<1$,
\begin{equation*}
\lvert f(a)\rvert^p\leq\frac{C}{\delta^n}\int_{E_\delta(a)}\lvert f(y)\rvert^p\,d\tau(y),
\end{equation*}
where $C=1$ if $p\geq 1$ and $C=2^{n/p}$
if $0<p<1$.
\end{lemma}

\begin{proposition}\label{Pfaminlp}
Let $\alpha>-1$, $0<p<\infty$, $0<r<1$ and $\{a_m\}$ be $r$-separated.
The operator $T\colon\mathcal B^p_\alpha\to\ell^p$ defined by
\begin{equation}\label{defineT}
Tf=\big\{f(a_m)(1-\lvert a_m\rvert^2)^{(\alpha+n)/p}\big\}
\end{equation}
is bounded.
\end{proposition}

\begin{proof}
Applying Lemma \ref{Lsubmean} with $\delta=r/2$ and noting that
$(1-\lvert y\rvert^2)\sim (1-\lvert a_m\rvert^2)$ for
$y\in E_{r/2}(a_m)$ by Lemma \eqref{Lratiosquare}, we obtain
\begin{equation*}
\lvert f(a_m)\rvert^p(1-\lvert a_m\rvert^2)^{\alpha+n}
\lesssim\int_{E_{r/2}(a_m)}\lvert f(y)\rvert^p\,d\nu_\alpha(y).
\end{equation*}
Since $E_{r/2}(a_m)$ are disjoint, we deduce
\begin{equation*}
\|Tf\|^p_{\ell^p}
=\sum_{m=1}^\infty\lvert f(a_m)\rvert^p(1-\lvert a_m\rvert^2)^{\alpha+n}
\lesssim\sum_{m=1}^\infty\int_{E_{r/2}(a_m)}\lvert f(y)\rvert^p\,d\nu_\alpha(y)
\leq \|f\|_{\mathcal B^p_\alpha}^p.\qedhere
\end{equation*}
\end{proof}

We need a slightly modified version of the above operator $T$.

\begin{proposition}\label{Patom2}
For $\alpha>-1$ and $0<p<\infty$, choose $s$ so that \eqref{scond} holds
and let $0<r<1$.
If $\{a_m\}$ is an $r$-lattice and $\{E_m\}$ is the associated sequence
as given in Lemma \ref{LEm}, then the operator
$\hat{T}\colon\mathcal B^p_\alpha\to\ell^p$ defined by
\begin{equation}\label{defineThat}
\hat{T}f
=\bigl\{f(a_m)(1-\lvert a_m\rvert^2)^{(\alpha+n)/p-(s+n)}\,\nu_s(E_m)\bigr\}.
\end{equation}
is bounded.
\end{proposition}

\begin{proof}
This follows from Proposition \ref{Pfaminlp} since by \eqref{Emsim},
$\nu_s(E_m)\sim (1-\lvert a_m\rvert^2)^{s+n}$.
\end{proof}

\begin{proposition}\label{Patom3}
For $\alpha>-1$ and $0<p<\infty$, choose $s$ so that \eqref{scond} holds.
There exists a constant $C>0$ depending only on $n,\alpha,p,s$ such that
if $\{a_m\}$ is an $r$-lattice with $r<1/8$, then
$\|I-U\hat{T}\|_{\mathcal B^p_\alpha\to\mathcal B^p_\alpha}\leq Cr$.
\end{proposition}

In the proofs of the previous propositions the constants
depended on the separation constant $r$.
In the proof of Proposition \ref{Patom3}, we will be careful
that the suppressed constants are independent of $r$.
However we will frequently use the fact that $r$ is bounded above by $1/8$.
We prove the cases $p\geq 1$ and $0<p<1$ separately.

\begin{proof}[Proof of Proposition \ref{Patom3} when $p\geq 1$]
Writing $\nu_s(E_m)=\int_{E_m}d\nu_s(y)$, by \eqref{defineU}
and \eqref{defineThat},
\begin{equation*}
 U\hat{T}f(x)=\sum_{m=1}^\infty \int_{E_m} f(a_m)\mathcal R_s(x,a_m)\,d\nu_s(y).
\end{equation*}
By \eqref{scond} and Lemma \ref{Lprojection} we have $P_s f=f$ and so
\begin{equation*}
 f(x)=\sum_{m=1}^\infty\int_{E_m} f(y)\mathcal R_s(x,y)\,d\nu_s(y).
\end{equation*}
Combining these shows
\begin{equation}\label{I-UT}
\begin{split}
  (I-U\hat{T})f(x)&
  =\sum_{m=1}^\infty
  \int_{E_m}\bigl(\mathcal R_s(x,y)-\mathcal R_s(x,a_m)\bigr)f(y)\,d\nu_s(y)\\
  &+\sum_{m=1}^\infty
  \int_{E_m}\mathcal R_s(x,a_m)\big(f(y)-f(a_m)\big)\,d\nu_s(y)
  =:h_1(x)+h_2(x).
\end{split}
\end{equation}
We first estimate $h_1$.
Let $y\in E_m$.
By the mean value theorem of calculus, there exists $\tilde{y}$ lying
on the line segment joining
$a_m$ and $y$ such that
\begin{equation*}
  \mathcal R_s(x,y)-\mathcal R_s(x,a_m)
  =\big\langle\, y-a_m,\nabla_2\mathcal R_s(x,\tilde{y})\,\big\rangle,
\end{equation*}
where $\nabla_2$ means the gradient is taken with respect to the
second variable.
Observe that, because $r$ is bounded above by $1/8$, there are
constants independent of $r$ such that for $y\in E_m\subset E_r(a_m)$,
by Lemma \ref{Lratiobracket}, $[y,a_m]\sim[y,y]=1-\lvert y\rvert^2$.
Thus by \eqref{pseudo},
\begin{equation*}
  \lvert y-a_m\rvert=\rho(y,a_m)[y,a_m]<r[y,a_m]
  \lesssim r(1-\lvert y\rvert^2).
\end{equation*}
Next, since $a_m$ and $y$ are both in the ball $E_r(a_m)$,
so is $\tilde{y}$.
Hence $\rho(y,\tilde{y})<1/4$ and by Lemma \ref{Lratiobracket},
$[x,y]\sim[x,\tilde{y}]$
for every $x\in\mathbb B$ with the constants again not depending on $r$.
Therefore by Lemma \ref{LEstR} (b) and the
symmetry of $\mathcal R_s(\cdot,\cdot)$,
\begin{equation*}
  \big\lvert\nabla_2\mathcal R_s(x,\tilde{y})\big\rvert
  \lesssim\frac{1}{[x,\tilde{y}]^{s+n+1}}\sim\frac{1}{[x,y]^{s+n+1}}.
\end{equation*}
Combining these we see that for $y\in E_m$ and $x\in\mathbb B$,
\begin{equation}\label{EstRdif}
\lvert\mathcal R_s(x,y)-\mathcal R_s(x,a_m)\rvert
\lesssim\frac{r(1-\lvert y\rvert^2)}{[x,y]^{s+n+1}}
\lesssim\frac{r}{[x,y]^{s+n}},
\end{equation}
where in the last inequality we use
$[x,y]\geq 1-\lvert y\rvert$ by \eqref{xybig}.
Thus
\begin{equation*}
\lvert h_1(x)\rvert
\lesssim r\sum_{m=1}^\infty\int_{E_m}\frac{\lvert f(y)\rvert}{[x,y]^{s+n}}\,d\nu_s(y)
=r \int_{\mathbb B}\frac{\lvert f(y)\rvert}{[x,y]^{s+n}}\,d\nu_s(y)
=r Q_s(\lvert f\rvert)(x),
\end{equation*}
and it follows from Lemma \ref{Lprojection} that
$\|h_1\|_{L^p_\alpha}\lesssim r\|f\|_{B^p_\alpha}$.

We now estimate $h_2$.
Let $y\in E_m$.
As above, we have $P_sf=f$, and so
\begin{equation*}
f(y)-f(a_m)=\int_{\mathbb B}\bigl(\mathcal R_s(y,z)-\mathcal R_s(a_m,z)\bigr)
f(z)\,d\nu_s(z).
\end{equation*}
Since $\mathcal R_s(\cdot,\cdot)$ is symmetric, by \eqref{EstRdif},
\begin{equation*}
\lvert\mathcal R_s(y,z)-\mathcal R_s(a_m,z)\rvert
\lesssim\frac{r}{[y,z]^{s+n}},
\end{equation*}
for all $z\in\mathbb B$ with the constants not depending on $r$.
Thus
\begin{equation*}
\lvert f(y)-f(a_m)\rvert
\lesssim r\int_{\mathbb B}\frac{\lvert f(z)\rvert}{[y,z]^{s+n}}\,d\nu_s(z)
=rQ_s(\lvert f\rvert)(y),
\end{equation*}
and so
\begin{equation*}
\lvert h_2(x)\rvert
\lesssim r \sum_{m=1}^\infty
\int_{E_m}\lvert\mathcal R_s(x,a_m)\rvert\,Q_s(\lvert f\rvert)(y)\,d\nu_s(y)
\lesssim r \sum_{m=1}^\infty
\int_{E_m}\frac{Q_s(\lvert f\rvert)(y)}{[x,a_m]^{s+n}}\,d\nu_s(y),
\end{equation*}
where in the last inequality we use Lemma \ref{LEstR} (a).
By Lemma \ref{Lratiosquare} again, we have $[x,a_m]\sim[x,y]$ for
$y\in E_m\subset E_r(a_m)$ since $r<1/8$.
Hence
\begin{equation*}
\lvert h_2(x)\rvert
\lesssim r \int_{\mathbb B}\frac{Q_s(\lvert f\rvert)(y)}{[x,y]^{s+n}}\,d\nu_s(y)
=r Q_s\circ Q_s(\lvert f\rvert)(x),
\end{equation*}
and it follows from Lemma \ref{Lprojection} that
$\|h_2\|_{L^p_\alpha}\lesssim r\|f\|_{B^p_\alpha}$.

We conclude that
$\|(I-U\hat{T})f\|_{\mathcal B^p_\alpha}\leq Cr\|f\|_{B^p_\alpha}$,
with $C$ depending only on $n,\alpha,p,s$.
This finishes the proof when $p\geq 1$.
\end{proof}

In order to prove the case $0<p<1$, we need to do some preparation.
The following inequality is proved in \cite[Theorem 4.7.4 part (b)]{St1}.

\begin{lemma}\label{Lgradf}
Let $0<p<\infty$ and $0<\delta<1/2$.
There exists a constant $C>0$ depending only on $n,p,\delta$ such that
for all $a\in\mathbb B$ and $f\in\mathcal H(\mathbb B)$,
\begin{equation*}
\lvert\nabla^h f(a)\rvert^p
\leq \frac{C}{\delta^n}\int_{E_\delta(a)}\lvert f(y)\rvert^p\,d\tau(y).
\end{equation*}
\end{lemma}

The next lemma is a special case of Theorem \ref{Tinclusion}
part (a).

\begin{lemma}\label{Linclude}
Let $0<p<1$ and $\alpha>-1$.
Then $\mathcal B^p_{\alpha}\subset\mathcal B^1_{(\alpha+n)/p-n}$
and the inclusion is continuous.
\end{lemma}

\begin{proof}
By \cite[Eqn.~(10.1.5)]{St1}, there exists a constant $C>0$ depending
only on $n,\alpha,p$ such that
\begin{equation}\label{pointeva}
\lvert f(x)\rvert
\leq\frac{C}{(1-\lvert x\rvert^2)^{(\alpha+n)/p}}
\|f\|_{\mathcal B^p_\alpha},
\end{equation}
for all $x\in\mathbb B$ and $f\in\mathcal B^p_\alpha$.
In the integral below writing
$\lvert f(x)\rvert=\lvert f(x)\rvert^p\lvert f(x)\rvert^{1-p}$ and
applying  \eqref{pointeva} to the factor $\lvert f(x)\rvert^{1-p}$, we deduce
\begin{align*}
\int_{\mathbb B}\lvert f(x)\rvert
(1-\lvert x\rvert^2)^{(\alpha+n)/p-n}\,d\nu(x)
&\leq C^{1-p}\|f\|_{\mathcal B^p_\alpha}^{1-p}
\int_{\mathbb B}\lvert f(x)\rvert^p(1-\lvert x\rvert^2)^\alpha
\,d\nu(x)\\
&=C^{1-p}\|f\|_{\mathcal B^p_\alpha}.\qedhere
\end{align*}
\end{proof}

\begin{proof}[Proof of Proposition \ref{Patom3} when $0<p<1$]
In this part of the proof we can not use the projection theorem
which requires $p\geq 1$.
Instead, we will follow \cite[p.~19]{CR} and use a suitable
rearrangement of the sequence $\{a_m\}$ as described below.

Pick a $1/2$-lattice $\{b_m\}$ and fix it throughout the proof.
Denote the associated sequence of sets as described in Lemma \ref{LEm}
to the lattice $\{b_m\}$ as $\{D_m\}$.
That is, the sets $D_m$ are disjoint,
$\bigcup_{m=1}^\infty D_m=\mathbb B$ and
\begin{equation*}
E_{1/4}(b_m)\subset D_m\subset E_{1/2}(b_m)
\qquad (m=1,2,\dots).
\end{equation*}
Given an $r$-lattice $\{a_m\}$ with $r<1/8$,
renumber $\{a_m\}$ in the following way.
Call the elements of $\{a_m\}$ that are in $D_1$ as
$a_{11},a_{12},\dots,a_{1\kappa_1}$ and in general call the points
of $\{a_m\}$ that are in $D_m$ as
$a_{m1},a_{m2},\dots,a_{m\kappa_m}$.
Denote the sets given in Lemma \ref{LEm} corresponding
to this renumbering as $E_{mk}$.
That is, $E_{mk}$ are disjoint,
$\bigcup_{m=1}^\infty\bigcup_{k=1}^{\kappa_m}E_{mk}=\mathbb B$ and
\begin{equation*}
E_{r/2}(a_{mk})\subset E_{mk}\subset E_r(a_{mk})
\qquad (m=1,2,\dots,\ k=1,2,\dots,\kappa_m).
\end{equation*}
By the above construction, since $a_{mk}\in D_m\subset E_{1/2}(b_m)$,
we have
\begin{equation}\label{rhoamkbm}
\rho(a_{mk},b_m)<1/2
\qquad (m=1,2,\dots,\ k=1,2,\dots,\kappa_m),
\end{equation}
and by the triangle inequality and the fact that $r<1/8$,
\begin{equation}\label{Emk}
E_{mk}\subset E_{5/8}(b_m).
\end{equation}

Suppose now $f\in\mathcal B^p_\alpha$.
We claim that $P_s f=f$.
This is true because by Lemma \ref{Linclude}, $f$ is
in $\mathcal B^1_{(\alpha+n)/p-n}$ and for this space the
required condition in Lemma \ref{Lprojection} is
$s>(\alpha+n)/p-n$ which holds by \eqref{scond}.
Therefore
\begin{equation*}
f(x)=\int_{\mathbb B}f(y)\mathcal R_s(x,y)\,d\nu_s(y)
=\sum_{m=1}^\infty\sum_{k=1}^{\kappa_m}
\int_{E_{mk}}f(y)\mathcal R_s(x,y)\,d\nu_s(y).
\end{equation*}
Next, with the above rearrangement, by \eqref{defineU} and
\eqref{defineThat},
\begin{equation*}
U\hat{T}f(x)=\sum_{m=1}^\infty\sum_{k=1}^{\kappa_m}
f(a_{mk})\nu_s(E_{mk})\mathcal R_s(x,a_{mk})
\end{equation*}
and so, similar to \eqref{I-UT}, we have
\begin{align*}
(I-U\hat{T})f(x)
&=\sum_{m=1}^\infty\sum_{k=1}^{\kappa_m}
\int_{E_{mk}}\bigl(\mathcal R_s(x,y)-\mathcal R_s(x,a_{mk})\bigr)f(y)\,d\nu_s(y)\\
&+\sum_{m=1}^\infty\sum_{k=1}^{\kappa_m}
\int_{E_{mk}}\mathcal R_s(x,a_{mk})\big(f(y)-f(a_{mk})\big)\,d\nu_s(y)
=:h_1(x)+h_2(x).
\end{align*}

We first estimate $h_1$.
We will again be careful in the estimates below that the suppressed constants
are independent of the separation constant $r$.
Let $y\in E_{mk}$.
First, as is shown in \eqref{EstRdif}, for all $x\in\mathbb B$,
\begin{equation}\label{Rdif}
\lvert\mathcal R_s(x,y)-\mathcal R_s(x,a_{mk})\rvert
\lesssim\frac{r}{[x,y]^{s+n}}
\lesssim\frac{r}{[x,b_m]^{s+n}},
\end{equation}
where in the last inequality we use Lemma \ref{Lratiobracket} and the fact that
$\rho(y,b_m)<5/8$ by \eqref{Emk}.
Next, applying Lemma \ref{Lsubmean} with $\delta=1/8$ and noting that
$E_{1/8}(y)\subset E_{3/4}(b_m)$, we obtain
\begin{equation}\label{fsmall}
\begin{split}
\lvert f(y)\rvert^p
\lesssim\int_{E_{1/8}(y)}\lvert f(z)\rvert^p\, d\tau(z)
&\leq\int_{E_{3/4}(b_m)}\lvert f(z)\rvert^p\, d\tau(z)\\
&\lesssim (1-\lvert b_m\rvert^2)^{-(\alpha+n)}
\int_{E_{3/4}(b_m)}\lvert f(z)\rvert^p\, d\nu_\alpha(z),
\end{split}
\end{equation}
where in the last inequality we use
$(1-\lvert z\rvert^2)\sim(1-\lvert b_m\rvert^2)$ for $z\in E_{3/4}(b_m)$
by Lemma \ref{Lratiosquare}.
Combining \eqref{Rdif} and \eqref{fsmall} we deduce
\begin{equation}\label{Esth1}
\lvert h_1(x)\rvert
\lesssim r\sum_{m=1}^\infty
\frac{(1-\lvert b_m\rvert^2)^{-(\alpha+n)/p}}{[x,b_m]^{s+n}}
\bigg(\int_{E_{3/4}(b_m)}\lvert f(z)\rvert^p d\nu_\alpha(z)\bigg)^{\frac{1}{p}}
\sum_{k=1}^{\kappa_m}\nu_s(E_{mk}).
\end{equation}
Since the sets $E_{mk}$ are disjoint and $E_{mk}\subset E_{5/8}(b_m)$
for every $k=1,\dots,\kappa_m$ by \eqref{Emk},
we have $\sum_{k=1}^{\kappa_m}\nu_s(E_{mk})\leq\nu_s(E_{5/8}(b_m))$.
Also $\nu_s(E_{5/8}(b_m))\sim(1-\lvert b_m\rvert^2)^{s+n}$ by Lemma
\ref{Lratiosquare} and \eqref{PseudoBall}.
Using this and Lemma \ref{Lseq} (i) shows
\begin{equation*}
\lvert h_1(x)\rvert^p
\lesssim r^p\sum_{m=1}^\infty
\frac{(1-\lvert b_m\rvert^2)^{p(s+n)-(\alpha+n)}}{[x,b_m]^{p(s+n)}}
\int_{E_{3/4}(b_m)}\lvert f(z)\rvert^p\, d\nu_\alpha(z).
\end{equation*}
Integrating over $\mathbb B$ with respect to $d\nu_\alpha$,
applying Fubini's theorem, and noting that
\begin{equation*}
(1-\lvert b_m\rvert^2)^{p(s+n)-(\alpha+n)}
\int_{\mathbb B}\frac{d\nu_\alpha(x)}{[x,b_m]^{p(s+n)}}
\lesssim 1,
\end{equation*}
by Lemma \ref{LIntonSandB} and \eqref{scond}, we obtain
\begin{equation}
\|h_1\|^p_{L^p_\alpha}
\lesssim r^p\sum_{m=1}^\infty
\int_{E_{3/4}(b_m)}\lvert f(z)\rvert^p\, d\nu_\alpha(z).
\end{equation}
Finally, by Lemma \ref{LNballs}, there exists $N$ such that every
$z\in\mathbb B$ belongs at most $N$ of the balls $E_{3/4}(b_m)$,
and so
$\sum_{m=1}^\infty
\int_{E_{3/4}(b_m)}\lvert f(z)\rvert^p\, d\nu_\alpha(z)
\leq N\int_{\mathbb B}\lvert f(z)\rvert^p\, d\nu_\alpha(z)$.
We conclude that
\begin{equation}\label{h1norm}
\|h_1\|^p_{L^p_\alpha}
\lesssim r^p \|f\|^p_{L^p_\alpha}.
\end{equation}

We next estimate $h_2$.
Again, let $y\in E_{mk}$.
By the mean-value theorem of calculus, there exists $\tilde{y}$
lying on the line segment joining $a_{mk}$ and $y$ such that
\begin{equation*}
f(y)-f(a_{mk})=\big\langle y-a_{mk},\nabla f(\tilde{y})\big\rangle.
\end{equation*}
Since $y\in E_{mk}\subset E_r(a_{mk})$, the point $\tilde{y}$
is also in the ball $E_r(a_{mk})$ and since $r<1/8$,
\begin{equation}\label{tildeyamk}
\rho(\tilde{y},a_{mk})<1/8.
\end{equation}
As before, by \eqref{pseudo},
\begin{equation*}
\lvert y-a_{mk}\rvert=\rho(y,a_{mk})[y,a_{mk}]<r[y,a_{mk}],
\end{equation*}
and since $\rho(y,a_{mk})<r<1/8$, we have
$[y,a_{mk}]\sim [a_{mk},a_{mk}]=(1-\lvert a_{mk}\rvert^2)$
by Lemma \ref{Lratiobracket}.
Therefore
\begin{equation*}
\lvert y-a_{mk}\rvert
\lesssim r(1-\lvert a_{mk}\rvert^2)
\sim r(1-\lvert\tilde{y}\rvert^2),
\end{equation*}
where in the last part we use Lemma \ref{Lratiosquare} and
\eqref{tildeyamk}.
This shows that
\begin{equation*}
\lvert f(y)-f(a_{mk})\rvert
\lesssim r(1-\lvert\tilde{y}\rvert^2)\lvert\nabla f(\tilde{y})\rvert
=r\lvert\nabla^h f(\tilde{y})\rvert,
\end{equation*}
by \eqref{nablahf}.
Next, applying Lemma \ref{Lgradf} with $\delta=1/8$ and then using
$E_{1/8}(\tilde{y})\subset E_{3/4}(b_m)$ which follows from
\eqref{tildeyamk} and \eqref{rhoamkbm}, we obtain
\begin{align*}
\lvert\nabla^h f(\tilde{y})\rvert^p
\lesssim\int_{E_{1/8}(\tilde{y})}\lvert f(z)\rvert^p\,d\tau(z)
&\leq\int_{E_{3/4}(b_m)}\lvert f(z)\rvert^p\,d\tau(z)\\
&\lesssim(1-\lvert b_m\rvert^2)^{-(\alpha+n)}
\int_{E_{3/4}(b_m)}\lvert f(z)\rvert^p\, d\nu_\alpha(z),
\end{align*}
similar to \eqref{fsmall}.
Using also that
\begin{equation*}
\lvert\mathcal R_s(x,a_{mk})\rvert
\lesssim\frac{1}{[x,a_{mk}]^{s+n}}
\sim\frac{1}{[x,b_m]^{s+n}},
\end{equation*}
which follows from Lemma \ref{LEstR} (a) and Lemma \ref{Lratiobracket}
with \eqref{rhoamkbm}, we conclude that
\begin{equation*}
\lvert h_2(x)\rvert
\lesssim r\sum_{m=1}^\infty
\frac{(1-\lvert b_m\rvert^2)^{-(\alpha+n)/p}}{[x,b_m]^{s+n}}
\bigg(\int_{E_{3/4}(b_m)}\lvert f(z)\rvert^p d\nu_\alpha(z)\bigg)^{\frac{1}{p}}
\sum_{k=1}^{\kappa_m}\nu_s(E_{mk}).
\end{equation*}
This estimate is same as \eqref{Esth1}.
Thus we again have
$\|h_2\|^p_{L^p_\alpha}\lesssim r^p\|f\|^p_{\mathcal B^p_\alpha}$
and hence
$\|(I-U\hat{T})f\|^p_{\mathcal B^p_\alpha}
\leq \|h_1\|^p_{L^p_\alpha}+\|h_2\|^p_{L^p_\alpha}
\lesssim r^p\|f\|^p_{\mathcal B^p_\alpha}$.
We conclude that
$\|(I-U\hat{T})\|\leq Cr$, where $C$ depends only
on $n,p,\alpha,s$.
\end{proof}

Proposition \ref{Patom3} immediately implies Theorem \ref{Tatomic}.

\begin{proof}[Proof of Theorem \ref{Tatomic}]
By Proposition \ref{Patom3}, if $r$ is small enough, then
$\|I-U\hat{T}\|<1$, and so $U\hat{T}$ has bounded inverse.
Given $f\in\mathcal B^p_\alpha$, if we let
$\lambda=\hat{T}(U\hat{T})^{-1}f$, then $\lambda\in\ell^p$,
$U\lambda=f$, and $\|\lambda\|_{\ell^p}\sim\|f\|_{\mathcal B^p_\alpha}$.
We note that in the equivalence
$\|\lambda\|_{\ell^p}\sim\|f\|_{\mathcal B^p_\alpha}$,
the suppressed constants depend also on $r$.
\end{proof}

To justify that one can replace \eqref{atomFrm} with \eqref{newad} in
Theorem \ref{Tatomic}, just change
$U\lambda$ in \eqref{defineU} with
$U\lambda(x)=\sum_{m=1}^\infty\lambda_m
\mathcal R_s(x,a_m)/\|\mathcal R_s(\cdot,a_m)\|_{\mathcal B^p_\alpha}$,
and $\hat{T}f$ in \eqref{defineThat} with
$\hat{T}f=\{f(a_m)\,\|\mathcal R_s(\cdot,a_m)\|_{\mathcal B^p_\alpha}\,\nu_s(E_m)\}$.
Then $U\hat{T}$ remains the same and so is Proposition \ref{Patom3}.
The changes in the proofs of Propositions \ref{Patom1} and \ref{Patom2} requires
only \eqref{IntRpower}.

\section{Interpolation}\label{SInterp}

To prove Theorem \ref{TInterp} we again consider two operators.
One is $\hat{U}\colon\ell^p\to\mathcal B^p_\alpha$, a slightly
modified version of $U$ given in \eqref{defineU} and the other is
$T\colon\mathcal B^p_\alpha\to\ell^p$,\begin{equation*}
Tf=\big\{f(a_m)(1-\lvert a_m\rvert^2)^{(\alpha+n)/p}\big\},
\end{equation*}
given in \eqref{defineT}.
Our main purpose is to show that the composition
$T\hat{U}:\ell^p\to\ell^p$ is invertible when the separation
constant is large enough.

\begin{proposition}\label{PInterp1}
For $\alpha>-1$ and $0<p<\infty$, choose $s$ so that \eqref{scond} holds
and let $0<r<1$.
If $\{a_m\}$ is $r$-separated, then the operator
$\hat{U}\colon\ell^p\to\mathcal B^p_\alpha$ mapping $\lambda=\{\lambda_m\}$ to
\begin{equation}\label{defineUhat}
\hat{U}\lambda(x)
=\sum_{m=1}^\infty\lambda_m(1-\lvert a_m\rvert^2)^{-(\alpha+n)/p}\,
\frac{\mathcal R_s(x,a_m)}{\mathcal R_s(a_m,a_m)}\qquad (x\in\mathbb B)
\end{equation}
is bounded.
The above series converges absolutely and uniformly on compact subsets of $\mathbb B$,
and also in $\|\cdot\|_{\mathcal B^p_\alpha}$.
\end{proposition}

Noting that $\mathcal R_s(a_m,a_m)\sim (1-\lvert a_m\rvert^2)^{-(s+n)}$
by \eqref{EstRtwosided}, this proposition can be proved in the same way
as Proposition \ref{Patom1}.
The minor changes required are omitted.

\begin{proposition}\label{PInterp2}
For $\alpha>-1$ and $0<p<\infty$, choose $s$ so that \eqref{scond} holds.
There exists $1/2<r_0<1$ depending only on $n,\alpha,p,s$
such that if $\{a_m\}$ is $r$-separated with $r>r_0$,
then $\|T\hat{U}-I\|_{\ell^p\to\ell^p}<1$.
\end{proposition}

This proposition immediately implies Theorem \ref{TInterp},
similar to the proof of Theorem \ref{Tatomic} above.

To verify Proposition \ref{PInterp2}, let $\lambda=\{\lambda_m\}\in\ell^p$.
Then the $m$-th component of the sequence $(T\hat{U}-I)\lambda$ is given by
\begin{equation*}
\{(T\hat{U}-I)\lambda\}_m
=(1-\lvert a_m\rvert^2)^{(\alpha+n)/p}
\sum_{\substack{k=1\\ k\ne m}}^\infty\lambda_k
(1-\lvert a_k\rvert^2)^{-(\alpha+n)/p}
\frac{\mathcal R_s(a_m,a_k)}{\mathcal R_s(a_k,a_k)},
\end{equation*}
and by Lemma \ref{LEstR} (a) and \eqref{EstRtwosided}, we have
\begin{equation}\label{mthcomp}
\bigl\lvert\{(T\hat{U}-I)\lambda\}_m\bigr\rvert
\leq C(1-\lvert a_m\rvert^2)^{(\alpha+n)/p}
\sum_{\substack{k=1\\ k\ne m}}^\infty\lvert\lambda_k\rvert\,
\frac{(1-\lvert a_k\rvert^2)^{s+n-(\alpha+n)/p}}{[a_m,a_k]^{s+n}},
\end{equation}
where the constant $C$ depends only on $n,\alpha,p$ and $s$.

To estimate the norm $\|(T\hat{U}-I)\lambda\|_{\ell^p}$, we need
an estimate of the series on the right of \eqref{mthcomp} (without
the $\lvert\lambda_k\rvert$ term) as given in Lemma \ref{LsmallSeries}
below.
We first prove this lemma and complete the proof of
Proposition \ref{PInterp2} at the end of the section.

Observe that by Lemma \ref{LIntonSandB}, for $b>-1$ and $c>0$,
there exists $C>0$ (depending only on $n,b,c$) such that
\begin{equation*}
(1-\lvert a\rvert^2)^c
\int_{\mathbb B}\frac{(1-\lvert y\rvert^2)^b}{[a,y]^{n+b+c}}\,d\nu(y)
\leq C,
\end{equation*}
uniformly for all $a\in\mathbb B$.
The next result will be needed in the proof of Lemma \ref{LsmallSeries}.

\begin{lemma}\label{LsmallInt}
Let $b>-1$ and $c>0$.
For $\varepsilon >0$, there exists $0<r_\varepsilon<1$ such that if $r_\varepsilon<r<1$,
then for all $a\in\mathbb B$,
\begin{equation*}
(1-\lvert a\rvert^2)^c
\int_{\mathbb B\backslash E_r(a)}\,
\frac{(1-\lvert y\rvert^2)^b}{[a,y]^{n+b+c}}\,d\nu(y)
<\varepsilon.
\end{equation*}
\end{lemma}

\begin{proof}
The change of variable $y=\varphi_a(z)$ with the facts that
$\varphi_a(\mathbb B_r)=E_r(a)$, and $\lvert J_{\varphi_a}\rvert$ is
as given in \eqref{Jacobian} shows
\begin{equation*}
K:=(1-\lvert a\rvert^2)^c
\int\limits_{\mathbb B\backslash E_r(a)}
\frac{(1-\lvert y\rvert^2)^b d\nu(y)}{[a,y]^{n+b+c}}
=(1-\lvert a\rvert^2)^c
\int\limits_{\mathbb B\backslash \mathbb B_r}
\frac{(1-\lvert \varphi_a(z)\rvert^2)^{b+n}\,d\nu(z)}
{[a,\varphi_a(z)]^{n+b+c}(1-\lvert z\rvert^2)^n}.
\end{equation*}
Applying \eqref{MobiusIdnt} and Lemma \ref{La_varphia} and
simplifying we deduce
\begin{equation*}
K=\int_{\mathbb B\backslash \mathbb B_r}
\frac{(1-\lvert z\rvert^2)^b}{[a,z]^{n+b-c}}\,d\nu(z)
=n\int_r^1 t^{n-1}(1-t^2)^b\int_{\mathbb S}
\frac{d\sigma(\zeta)}{\lvert ta-\zeta\rvert^{n+b-c}}\,dt,
\end{equation*}
where in the second equality we integrate in polar coordinates and
use the fact that $[a,t\zeta]=\lvert ta-\zeta\rvert$ by \eqref{bracketxy}.
By Lemma \ref{LIntonSandB} and the inequality $1-\lvert a\rvert^2t^2\geq 1-t^2$,
we have
\begin{equation*}
\int_{\mathbb S} \frac{d\sigma(\zeta)}{\lvert ta-\zeta\rvert^{n+b-c}}
\leq C g(t):=
\begin{cases}
\dfrac{1}{(1-t^2)^{1+b-c}},&\text{if $1+b-c>0$};\\
1+\log\dfrac{1}{1-t^2},&\text{if $1+b-c=0$};\\
1,&\text{if $1+b-c<0$},
\end{cases}
\end{equation*}
where the constant $C$ depends only on $n,b,c$.
Thus $K\leq Cn\int_r^1 t^{n-1}(1-t^2)^b\,g(t)\,dt$ and
since in all the three cases the integral
$\int_0^1 t^{n-1}(1-t^2)^b\,g(t)\,dt$ is finite because $b>-1$ and $c>0$,
one can make $K<\varepsilon$ by choosing $r$ sufficiently close to $1$.
\end{proof}

The next lemma is an analogue of Lemma 3.1 of \cite{Ro}.

\begin{lemma}\label{LsmallSeries}
Let $b>n-1$ and $c>0$.
For $1/2<r<1$, there exists $C(r)>0$ (depending also on $n,b$ and $c$)
such that if $\{a_m\}$ is $r$-separated, then
\begin{equation*}
(1-\lvert a_m\rvert^2)^c
\sum_{\substack{k=1\\ k\ne m}}^\infty\frac{(1-\lvert a_k\rvert^2)^b}{[a_m,a_k]^{b+c}}
\leq C(r),
\end{equation*}
for every $m=1,2,\dots$.
Moreover, one can choose $C(r)$ to be arbitrarily small by making $r$
sufficiently close to $1$.
\end{lemma}

\begin{proof}
By the Lemmas \ref{Lratiosquare}, \ref{Lratiobracket} and \eqref{PseudoBall},
there exists $C>0$ depending only on $n,b,c$ such that
\begin{equation*}
\frac{(1-\lvert a\rvert^2)^b}{[x,a]^{b+c}}
\leq C\int_{E_{1/4}(a)}\frac{(1-\lvert y\rvert^2)^{b-n}}{[x,y]^{b+c}}\,d\nu(y),
\end{equation*}
for all $a,x\in\mathbb B$.
If $\{a_m\}$ is $r$-separated with $r>1/2$, then the balls $E_{1/4}(a_m)$
are disjoint and therefore
\begin{equation*}
(1-\lvert a_m\rvert^2)^c
\sum_{\substack{k=1\\ k\ne m}}^\infty\frac{(1-\lvert a_k\rvert^2)^b}{[a_m,a_k]^{b+c}}
\leq C(1-\lvert a_m\rvert^2)^c
\int_{\bigcup^\infty_{\substack{k=1\\ k\ne m}}E_{1/4}(a_k)}
\frac{(1-\lvert y\rvert^2)^{b-n}}{[a_m,y]^{b+c}}\,d\nu(y).
\end{equation*}
Set
\begin{equation*}
  R:=\frac{r-1/4}{1-r/4}.
\end{equation*}
Clearly, $0<R<1$.
We claim that $\bigcup^\infty_{\substack{k=1\\ k\ne m}}E_{1/4}(a_k)
\subset\mathbb B\backslash E_R(a_m)$.
To see this note first that for $0\leq t_0<1$, the function
$t\mapsto (t+t_0)/(1+tt_0)$ is increasing on the interval $0\leq t<1$.
Therefore by the strong triangle inequality in Lemma \ref{LstrongTri}, if
$z\in E_{1/4}(a_k)$ with $k\ne m$, then
\begin{equation*}
r\leq
\rho(a_k,a_m)\leq\frac{\rho(a_k,z)+\rho(z,a_m)}{1+\rho(a_k,z)\rho(z,a_m)}
\leq\frac{1/4+\rho(z,a_m)}{1+\rho(z,a_m)/4},
\end{equation*}
which implies $\rho(z,a_m)\geq R$.
Thus
\begin{equation*}
(1-\lvert a_m\rvert^2)^c
\sum_{\substack{k=1\\ k\ne m}}^\infty
\frac{(1-\lvert a_k\rvert^2)^b}{[a_m,a_k]^{b+c}}
\leq C(1-\lvert a_m\rvert^2)^c
\int_{\mathbb B\backslash E_R(a_m)}
\frac{(1-\lvert y\rvert^2)^{b-n}}{[a_m,y]^{b+c}}\,d\nu(y),
\end{equation*}
and since $R\to 1^-$ as $r\to1^-$, the desired result follows
from Lemma \ref{LsmallInt}.
\end{proof}

We now complete the proof of Proposition \ref{PInterp2}.
We consider the cases $0<p\leq 1$ and $p>1$ separately.

\begin{proof}[Proof of Proposition \ref{PInterp2} when $0<p\leq 1$]
For $\lambda=\{\lambda_m\}\in\ell^p$, by \eqref{mthcomp},
Lemma \ref{Lseq} (i) and Fubini's theorem,
\begin{align*}
\|(T\hat{U}-I)\lambda\|_{\ell^p}^p
&\leq C^p\sum_{m=1}^\infty(1-\lvert a_m\rvert^2)^{\alpha+n}
\Biggl(\sum_{\substack{k=1\\ k\ne m}}^\infty\lvert\lambda_k\rvert\,
\frac{(1-\lvert a_k\rvert^2)^{s+n-(\alpha+n)/p}}{[a_m,a_k]^{s+n}}
\Biggr)^p\\
&\leq C^p\sum_{m=1}^\infty(1-\lvert a_m\rvert^2)^{\alpha+n}
\sum_{\substack{k=1\\ k\ne m}}^\infty\lvert\lambda_k\rvert^p\,
\frac{(1-\lvert a_k\rvert^2)^{p(s+n)-(\alpha+n)}}{[a_m,a_k]^{p(s+n)}}\\
&=C^p\sum_{k=1}^\infty\lvert\lambda_k\rvert^p
(1-\lvert a_k\rvert^2)^{p(s+n)-(\alpha+n)}
\sum_{\substack{m=1\\ m\ne k}}^\infty\,
\frac{(1-\lvert a_m\rvert^2)^{\alpha+n}}{[a_m,a_k]^{p(s+n)}}.
\end{align*}
By Lemma \ref{LsmallSeries}, there exists $C(r)$ such that
(note that $\alpha+n>n-1$ since $\alpha>-1$, and
$p(s+n)-(\alpha+n)>0$ by \eqref{scond})
\begin{equation*}
\|(T\hat{U}-I)\lambda\|_{\ell^p}^p\leq C^p\,C(r)\|\lambda\|_{\ell^p}^p.
\end{equation*}
Since $C(r)$ can be made arbitrarily small by making $r$ close enough
to $1$, the proposition follows.
\end{proof}

We next deal with the case $1<p<\infty$.
Let $p'$ be the conjugate exponent of $p$.
We employ Schur's test which, for the sequence space
$\ell^p$, has the following form (see \cite[Lemma 3.2]{Ro}):
Let $A=(A_{mk})_{1\leq m,k<\infty}$ be an infinite matrix with
nonnegative entries and $L_A:\ell^p\to\ell^p$ be the corresponding
operator taking $\lambda=\{\lambda_m\}$ to
\begin{equation*}
  \{L_A\lambda\}_m=\sum_{k=1}^\infty A_{mk}\lambda_k, \qquad m=1,2,\dots.
\end{equation*}
If there exist a constant $C>0$ and a positive sequence $\{\gamma_m\}$
such that
\begin{equation*}
\sum_{k=1}^\infty A_{mk}\gamma_k^{p'}\leq C\gamma_m^{p'},\qquad m=1,2,\dots,
\end{equation*}
and
\begin{equation*}
\sum_{m=1}^\infty A_{mk}\gamma_m^{p}\leq C\gamma_k^{p},\qquad k=1,2,\dots,
\end{equation*}
then the operator $L_A$ is bounded and $\|L_A\|_{\ell^p\to\ell^p}\leq C$.

\begin{proof}[Proof of Proposition \ref{PInterp2} when $1<p<\infty$]
Without loss of generality we can assume that the $r$-separated sequence
$\{a_m\}$ is maximal, that is $\{a_m\}$ is an $r$-lattice and so
is an infinite sequence.

For $m,k=1,2,\dots$, set $A_{mk}=0$ if $k=m$ and
\begin{equation*}
A_{mk}=(1-\lvert a_m\rvert^2)^{(\alpha+n)/p}\,
\frac{(1-\lvert a_k\rvert^2)^{s+n-(\alpha+n)/p}}{[a_m,a_k]^{s+n}},
\end{equation*}
if $k\ne m$.
Let $A=(A_{mk})$ and $L_A\colon\ell^p\to\ell^p$ be the corresponding operator.
Then by \eqref{mthcomp},
\begin{equation*}
\bigl\lvert\{(T\hat{U}-I)\lambda\}_m\bigr\rvert\leq C\{L_A\lambda\}_m.
\end{equation*}

To estimate $\|L_A\|$ with the Schur's test we take
$\gamma_m=(1-\lvert a_m\rvert^2)^{(n-1)/pp'}$.
Then
\begin{equation*}
\sum_{k=1}^\infty A_{mk}\gamma_k^{p'}
=(1-\lvert a_m\rvert^2)^{(\alpha+n)/p}\,
\sum_{\substack{k=1\\ k\ne m}}^\infty\,
\frac{(1-\lvert a_k\rvert^2)^{s+n-(\alpha+1)/p}}{[a_m,a_k]^{s+n}},
\end{equation*}
and by Lemma \ref{LsmallSeries}, there exists $C_1(r)$ such that
(we check that $s+n-(\alpha+1)/p>n-1$ by \eqref{scond},
and $(\alpha+1)/p>0$)
\begin{equation*}
\sum_{k=1}^\infty A_{mk}\gamma_k^{p'}
\leq (1-\lvert a_m\rvert^2)^{(\alpha+n)/p}
\frac{C_1(r)}{(1-\lvert a_m\rvert^2)^{(\alpha+1)/p}}=C_1(r)\gamma_m^{p'}.
\end{equation*}
Observe next that
\begin{equation*}
\sum_{m=1}^\infty A_{mk}\gamma_m^{p}
=(1-\lvert a_k\rvert^2)^{s+n-(\alpha+n)/p}\,
\sum_{\substack{m=1\\ m\ne k}}^\infty\,
\frac{(1-\lvert a_m\rvert^2)^{(\alpha+n)/p+(n-1)/p'}}{[a_m,a_k]^{s+n}}.
\end{equation*}
To apply Lemma \ref{LsmallSeries} we check that
$(\alpha+n)/p+(n-1)/p'=(\alpha+1)/p+n-1>n-1$, and
$s+n-\bigl((\alpha+n)/p+(n-1)/p'\bigr)=s+1-(\alpha+1)/p>0$
by \eqref{scond}.
Thus there exists $C_2(r)$ such that
\begin{equation*}
\sum_{m=1}^\infty A_{mk}\gamma_m^{p}
\leq (1-\lvert a_k\rvert^2)^{s+n-(\alpha+n)/p}
\frac{C_2(r)}{(1-\lvert a_k\rvert^2)^{s+n-(\alpha+n)/p-(n-1)/p'}}
=C_2(r)\gamma_k^p.
\end{equation*}

We deduce that $L_A$ is bounded and $\|L_A\|\leq\max\{C_1(r),C_2(r)\}$.
Therefore $\|T\hat{U}-I\|\leq C\max\{C_1(r),C_2(r)\}$ and
since both $C_1(r)$ and $C_2(r)$ can be made arbitrarily small by making
$r$ close enough to $1$, we conclude that
$\|T\hat{U}-I\|$ can be made small.
This finishes the proof of Proposition \ref{PInterp2}.
\end{proof}

\begin{remark}
Let $R_{\alpha}(x,y)$ be the reproducing kernel of the (Euclidean)
harmonic Bergman space $b^p_\alpha$.
In the above proof of the interpolation theorem, the only facts we
used about the $\mathcal H$-harmonic reproducing kernels
$\mathcal R_\alpha$ are the estimates given in
Lemma \ref{LEstR}, \eqref{EstRtwosided} and \eqref{IntRpower}.
However, it is well-known that exactly same estimates hold also for
the kernels $R_\alpha$.
It is also well-known that a sub-mean value inequality similar to
Lemma \ref{Lsubmean} holds for harmonic functions.
Thus, the above proof, without any change, works also in the harmonic
case and the interpolation theorem is valid if $B^p_\alpha$
is replaced with $b^p_\alpha$.
\end{remark}

\section{Inclusion Relations}

In this section we prove Theorem \ref{Tinclusion}.

\begin{proof}[Proof of Theorem \ref{Tinclusion}]
We first prove part (a).
Suppose $\mathcal B^p_\alpha\subset\mathcal B^q_\beta$.
Since point evaluations are bounded on Bergman spaces, by the
closed graph theorem, the inclusion
$i\colon\mathcal B^p_\alpha\to\mathcal B^q_\beta$ is continuous.
For every $s>-1$ and $a\in\mathbb B$, the function
$\mathcal R_s(a,\cdot)$ is bounded on $\mathbb B$ by
Lemma \ref{LEstR} (a) and \eqref{xybig}, so belongs to every
Bergman space.
By \eqref{IntRpower}, for large enough $s$, we have
\begin{equation*}
\frac{\|\mathcal R_s(a,\cdot)\|_{\mathcal B^q_\beta}}
{\|\mathcal R_s(a,\cdot)\|_{\mathcal B^p_\alpha}}
\sim(1-\lvert a\rvert^2)^{(\beta+n)/q-(\alpha+n)/p},
\end{equation*}
and the right-hand side is bounded as $\lvert a\rvert\to1^-$
only if $(\beta+n)/q\geq(\alpha+n)/p$.

Suppose now that
\begin{equation}\label{condInca}
\frac{\alpha+n}{p}\leq\frac{\beta+n}{q}.
\end{equation}
Pick $s$ large enough so that \eqref{scond} holds both for
$\alpha,p$ and $\beta,q$.
Let $\{a_m\}$ be an $r$-lattice with $r<r_0$, where
$r_0$ is as asserted in the atomic decomposition theorem.
Then for every $f\in\mathcal B^p_\alpha$, there exists
$\{\lambda_m\}\in\ell^p$ with
$\|\{\lambda_m\}\|_{\ell^p}\sim\|f\|_{\mathcal B^p_\alpha}$
such that
\begin{align*}
f(x)&=\sum_{m=1}^\infty\lambda_m
(1-\lvert a_m\rvert^2)^{s+n-(\alpha+n)/p}\,
\mathcal R_s(x,a_m)\\
&=\sum_{m=1}^\infty\lambda_m
(1-\lvert a_m\rvert^2)^{(\beta+n)/q-(\alpha+n)/p}\,
(1-\lvert a_m\rvert^2)^{s+n-(\beta+n)/q}\,
\mathcal R_s(x,a_m).
\end{align*}
The sequence
$\{\beta_m\}
=\bigl\{\lambda_m(1-\lvert a_m\rvert^2)^{(\beta+n)/q-(\alpha+n)/p}\bigr\}$
is also in $\ell^p$ by \eqref{condInca} and so $\{\beta_m\}\in\ell^q$ by
Lemma \ref{Lseq} (i).
It follows from Proposition \ref{Patom1} that $f\in\mathcal B^q_\beta$
and $\|f\|_{\mathcal B^q_\beta}\lesssim\|\{\beta_m\}\|_{\ell^q}
\leq\|\{\beta_m\}\|_{\ell^p}\leq\|\{\lambda_m\}\|_{\ell^p}
\lesssim\|f\|_{\mathcal B^p_\alpha}$.

We next prove part (b).
Note first that in this case $p/q>1$ and the conjugate
exponent of $p/q$ is $p/(p-q)$.

To see the if part, suppose
\begin{equation}\label{condIncb}
\frac{\alpha+1}{p}<\frac{\beta+1}{q}.
\end{equation}
By H\"{o}lder's inequality,
\begin{equation*}
\int_{\mathbb B}\lvert f(x)\rvert^q\,d\nu_\beta(x)
\leq\biggl(\int_{\mathbb B}\lvert f(x)\rvert^p\,d\nu_\alpha(x)\biggr)^{q/p}
\biggl(\int_{\mathbb B} (1-\lvert x\rvert^2)^{(\beta-\alpha\frac{q}{p})\frac{p}{p-q}}
\,d\nu(x)\biggr)^{(p-q)/p},
\end{equation*}
and since the exponent $(\beta-\alpha\frac{q}{p})\frac{p}{p-q}>-1$ by \eqref{condIncb},
we obtain $\|f\|_{\mathcal B^q_\beta}\lesssim \|f\|_{\mathcal B^p_\alpha}$.

Suppose now that $\mathcal B^p_\alpha\subset\mathcal B^q_\beta$.
Let $r_0$ be as asserted in the interpolation theorem
for $\mathcal B^p_\alpha$ and let $\{a_m\}$ be an
$r$-lattice with $r>r_0$.
Then for every $\{\lambda_m\}\in\ell^p$ there exists a function
$f\in\mathcal B^p_\alpha$ such that
\begin{equation*}
f(a_m)=\lambda_m(1-\lvert a_m\rvert^2)^{-(\alpha+n)/p}.
\end{equation*}
Since $f$ is also in $\mathcal B^q_\beta$ the sequence
$\{f(a_m)(1-\lvert a_m\rvert^2)^{(\beta+n)/q}\}$ is in $\ell^q$
by Proposition \ref{Pfaminlp}, and so
\begin{equation*}
\sum_{m=1}^\infty \lvert\lambda_m\rvert^q
(1-\lvert a_m\rvert^2)^{(\beta+n)-(\alpha+n)q/p}<\infty.
\end{equation*}
This shows that the series
$\sum_{m=1}^\infty \lvert\beta_m\rvert
(1-\lvert a_m\rvert^2)^{(\beta+n)-(\alpha+n)q/p}$ converges
for every $\{\beta_m\}\in\ell^{p/q}$.
By the duality in Lemma \ref{Lseq} (ii), we conclude that the
sequence $\{(1-\lvert a_m\rvert^2)^{(\beta+n)-(\alpha+n)q/p}\}$
is in $\ell^{p/(p-q)}$ and by Lemma \ref{LSumam} (b), this holds
only if
\begin{equation*}
\Bigl((\beta+n)-(\alpha+n)\frac{q}{p}\Bigr)\frac{p}{p-q}>n-1,
\end{equation*}
which is equivalent to \eqref{condIncb}.
\end{proof}


\begin{thebibliography}{10}
%
\bibitem{Ah}
L. A. Ahlfors,
M\"obius Transformations in Several Dimensions,
Univ. Minnesota, Minneapolis, 1981.
%
\bibitem{Amr}
\'{E}. Amar,
\textit{Suites d'interpolation pour les classes de Bergman de la boule et du
polydisque de $\mathbb C^n$},
Canadian J. Math. \textbf{30} (1978) no.4, 711--737.
%
\bibitem{CKL}
B. R. Choe, H. Koo, Y. J. Lee,
\textit{Positive Schatten class Toeplitz operators on the ball},
Studia Math. \textbf{189} (2008) 65--90.
%
\bibitem{CL}
B. R. Choe, Y. J. Lee,
\textit{Note on atomic decompositions of harmonic Bergman functions},
in Complex Analysis and its Applications, OCAMI Stud., vol. 2,
Osaka Munic. Univ. Press, Osaka, 2007, 11-24.
%
\bibitem{CY}
B. R. Choe, H. Yi,
\textit{Representations and interpolations of harmonic Bergman functions
on half-spaces},
Nagoya Math. J. \textbf{151} (1998) 51--89.
%
\bibitem{CR}
R. R. Coifman, R. Rochberg,
\textit{Representation theorems for holomorphic and harmonic functions in $L^p$},
Ast\'erisque \textbf{77} (1980) 12--66.
%
\bibitem{J}
P. Jaming,
Harmonic functions on the real hyperbolic ball I.
Boundary values and atomic decomposition of Hardy spaces,
Colloq. Math. \textbf{80} (1999) 63--82.
%
\bibitem{JMT}
M. Jevti\'{c}, X. Massaneda, P. J. Thomas,
\textit{Interpolating sequences for weighted Bergman spaces
of the ball},
Michigan Math. J. \textbf{43} (1996) 495--517.
%
\bibitem{LS}
C. W. Liu, J. H. Shi,
\textit{Invariant mean-value property and $\mathcal M$-harmonicity in the
unit ball of $\mathbb R^n$},
Acta Math. Sin. \textbf{19} (2003) 187--200.
%
\bibitem{Ro}
R. Rochberg,
\textit{Interpolation by functions in Bergman spaces},
Michigan Math. J. \textbf{29} (1982) 229--236.
%
\bibitem{RW}
J. Ryll, P. Wojtaszczyk,
\textit{On homogeneous polynomials on a complex ball},
Trans. Amer. Math. Soc. \textbf{276} (1983) 107--116.
%
\bibitem{St1}
M. Stoll,
Harmonic and Subharmonic Function Theory on the Hyperbolic Ball,
London Math. Soc. Lect. Note Series, vol. 431,
Cambridge University Press, Cambridge, 2016.
%
\bibitem{T1}
K. Tanaka,
\textit{Atomic decomposition of harmonic Bergman functions},
Hiroshima Math. J. \textbf{42} (2012) 143--160.
%
\bibitem{T2}
K. Tanaka,
\textit{Representation theorem for harmonic Bergman and Bloch functions},
Osaka J. Math. \textbf{50} (2013) 947--961.
%
\bibitem{U}
A. E. \"Ureyen,
\textit{$\mathcal H$-Harmonic Bergman projection on the real hyperbolic ball},
J. Math. Anal. Appl. \textbf{519} (2023) 126802.
\end{thebibliography}
\end{document}